\documentclass[12pt, reqno]{amsart}
\usepackage{amssymb, graphicx}
\usepackage{amsmath,amsthm,amssymb,mathrsfs,amsfonts}
\usepackage{hyperref}
\usepackage{a4wide}

\usepackage{xcolor}
\usepackage{float}
\usepackage{pgf,tikz}
\usetikzlibrary{arrows}
\def \C {\mathbb{C}}
\def \D {\mathbb{D}}
\def \R {\mathbb{R}}
\def \T {\mathbb{T}}

\def \ss  {\mathcal{S}}

\def\d{\mathrm{d}}

\newtheorem{theorem}{Theorem}[section]
\newtheorem{lemma}[theorem]{Lemma}

\newtheorem{remark}[theorem]{Remark}

\newtheorem{corollary}[theorem]{Corollary}

	\definecolor{uuuuuu}{rgb}{0.26666666666666666,0.26666666666666666,0.26666666666666666}
	\definecolor{ffqqqq}{rgb}{1.,0.,0.}
	\definecolor{qqqqff}{rgb}{0.,0.,1.}
	\definecolor{ffqqqq}{rgb}{1.,0.,0.}

\begin{document}

\title{Phase Retrieval for Wide Band Signals}
\author[Ph. Jaming, K. Kellay, R. Perez]{Philippe Jaming, Karim Kellay \& Rolando Perez III}
\address{Ph. Jaming, Universit\'{e} de Bordeaux, CNRS, Bordeaux INP, IMB, UMR 5251, F-33400, Talence, France}  
\email{Philippe.Jaming@math.u-bordeaux.fr}
\address{K. Kellay, Universit\'{e} de Bordeaux, CNRS, Bordeaux INP, IMB, UMR 5251, F-33400, Talence, France}  
\email{kkellay@math.u-bordeaux.fr}
\address{R. Perez III, Universit\'{e} de Bordeaux, CNRS, Bordeaux INP, IMB, UMR 5251, F-33400, Talence, France}
\address{Institute of Mathematics, University of the Philippines Diliman, 1101 Quezon City, Philippines}
\email{roperez@u-bordeaux.fr}
\keywords{phase retrieval, Hardy spaces}
\subjclass[2010]{30D05, 30H10, 42B10, 94A12}

\begin{abstract}
This study investigates the phase retrieval problem for wide-band signals. We solve the following problem: given 
	$f\in L^2(\R)$ with Fourier transform in $L^2(\R,e^{2c|x|}\,\mbox{d}x)$, we find all functions $g\in L^2(\R)$ 
	with Fourier transform in $L^2(\R,e^{2c|x|}\,\mbox{d}x)$, such that $|f(x)|=|g(x)|$ for all $x\in \R$. To do so, we first translate the problem to functions in the Hardy spaces on the disc via a conformal bijection, and take advantage of the inner-outer factorization. We also consider the same problem with additional constraints involving some transforms of $f$ and $g$, and determine if these constraints force uniqueness of the solution. 
\end{abstract}
\maketitle 

\section{Introduction}

The phase retrieval problem refers to the recovery of the phase of a function $f$ using given data on its magnitude $|f|$ and a priori assumptions on $f$. These problems are widely studied because of their physical applications in which the quantities involved are identified by their magnitude and phase, where the phase is difficult to measure while the magnitude is easily obtainable. Some physical applications of phase retrieval problems include works related to astronomy \cite{Da}, lens design \cite{Do}, x-ray crystallography \cite{Mi}, inverse scattering \cite{Sa}, and optics \cite{Se}. More physical examples were given in the survey articles of Luke {\it et. al.} \cite{LBL}
Klibanov {\it et. al.} \cite{Kli} and the book of Hurt \cite{Hu}. 
A more recent overview of the phase retrieval problem was given by the article of Grohs {\it et. al.} \cite{GKR}, which discussed a more general formulation of the phase retrieval problem using Banach spaces and bounded linear functionals, as well as results related to the uniqueness and stability properties of the problem.

Phase retrieval problems have recently been given more interest because of progress in the discrete (finite-dimensional) case, starting with the work of Cand\`es {\it et. al.} \cite{CLS}
and of Waldspurger {\it et. al.} \cite{WAM}, which formulated the phase retrieval problem as an optimization problem
(though the investigation of optimization tools in phase retrieval algorithms seems
to go back at least to the work of Luke {\it et. al.} \cite{BCL,BL,LBL}). On the other hand, phase retrieval problems devoted to the continuous (infinite-dimensional) case have been solved in various settings, such as, for one-dimensional band-limited functions \cite{Aku1,Aku2,Wa},
for functions in the Hardy space on the disc without singular parts \cite{Bo}, for real-valued band-limited functions from the absolute values of their samples \cite{Th}, for 2$\pi$-periodic time-limited signals from magnitude values on the unit circle \cite{FL}, and for real-valued functions in the Sobolev spaces \cite{Han}. We refer the reader to \cite{ADGY} for more discussion and examples of continuous phase retrieval problems, in particular on the stability of the problem and other useful references. Our aim in this paper is to investigate the phase retrieval problem for \textit{wide-band} functions, namely functions with mildly decreasing Fourier transforms.

Before we summarize our results, let us give a quick overview of the phase retrieval problem in the band-limited and narrow band cases: given a band-limited function (i.e. a function with compactly supported Fourier transform) $f\in L^2(\R)$,
find all band-limited functions $g\in L^2(\R)$ such that 
\begin{equation}
|f(x)|=|g(x)|,\qquad x\in\R.
\label{eq:band1}
\end{equation}
This problem in the class of compactly supported functions has been solved by Akutowicz \cite{Aku1,Aku2} in the mid-1950's, and independently by Walther \cite{Wa} in 1963. To solve the problem, they first used the Paley-Wiener Theorem which states that 
$f$ and $g$ extend into holomorphic functions in the complex plane
that are of exponential type, that is, $f,g$ grow like $e^{a|z|}$.
Next, they showed that \eqref{eq:band1} is then equivalent to
\begin{equation}
f(z)\overline{f(\bar z)}=g(z)\overline{g(\bar z)},  \qquad z\in\C.
\label{eq:band2}
\end{equation}
Observe that \eqref{eq:band2} is a reformulation of \eqref{eq:band1} when $z$ is {\em real} and is an equality between two holomorphic functions so that it is valid for all $z\in\C$.
Finally, they used the Hadamard Factorization Theorem which states that holomorphic functions of exponential type are characterized by their zeros. Now, \eqref{eq:band2} implies that each zero of $g$ is either a zero of $f$ or a complex conjugate of such a zero. Thus, it follows that $g$ can be obtained by changing arbitrarily many zeros of $f$ into their complex conjugates in the Hadamard factorization of $g$, and this is called {\em zero-flipping}.

McDonald \cite{Mc} extended this proof to functions that have Fourier transforms with very fast decrease at infinity. For instance, in the case of Gaussian decrease, if $
|\widehat{f}(\xi)|,|\widehat{g}(\xi)|\lesssim e^{-a|\xi|^2}$, $a>0$, then $f,g$ extend to holomorphic functions of exponential type 2 so that Hadamard factorization can still be used. Thus, the solutions also can be obtained by zero-flipping. Furthermore, this proof extends to functions which satisfy an exponential decay condition of the form
\begin{equation}
\label{eq:decay}
|\widehat{f}(\xi)|\lesssim e^{-a|\xi|^\alpha}\qquad \text{ and }\qquad |\widehat{g}(\xi)|\lesssim e^{-a|\xi|^\alpha}
\end{equation}
for some $a>0$ and $\alpha>1$,
but breaks down at $\alpha=1$. Therefore, the main goal of this work is to investigate the phase retrieval problem for functions satisfying \eqref{eq:decay} but for $\alpha=1$, i.e.
$|\widehat f(\xi)| |\lesssim e^{-a|\xi|}$ and $|\widehat g(\xi)|\lesssim e^{-a|\xi|}$. Functions with this decay are sometimes called wide-band signals in the engineering community
while those with a decay like \eqref{eq:decay} for $\alpha>1$ are said to be narrow-banded. Here, the functions $f$ and $g$ only extend holomorphically to an horizontal strip
$\mathcal{S}_a=\{z\in\C\,:|\text{Im} z|<a\}$ in the complex plane so that
\eqref{eq:band2} only holds for $z\in \mathcal{S}_a$, which implies that Hadamard factorization cannot be used. To 
overcome this difficulty, we first reduce the problem to the Hardy space on the disc using a conformal bijection. 
Recall that the Hardy spaces are well-known to be defined on the unit disc and the upper half plane (see for example 
\cite{Du,Ga,Ko,Ma,SW}). While it is possible to work directly on a corresponding Hardy space on the strip, we have 
decided that the proofs can be made more familiar on the unit disc. One of the reasons for working on a more familiar 
Hardy space is that the expression for the Poisson kernel is sufficiently simple to work with. Due to the conformal 
mapping, the expressions of the Poisson kernel in the other domains can be more complicated even in the simplest 
situations like the strip. By transferring the problem to the Hardy space on the unit disc, we then exploit the 
inner-outer-Blashcke factorization in this space. The solution is now more evolved than the band-limited case as, aside from zero-flipping, the singular inner function and the outer function also take part in the solution. For our main results in Corollary \ref{maincoro}, we go back to the strip to solve our initial problem using an equivalent inner-outer-Blaschke factorization on the strip, and provide analogs of the consequences shown on the phase retrieval problem on the disc.

Generally speaking, the solution set of a phase retrieval problem is very large. Thus, additional constraints are considered to reduce the solution set. For instance, Klibanov {\it et. al.} provided different 
examples of supplementary information to force uniqueness, or at least to reduce the solution set. 
We here follow the same aim of reducing the set of solutions by coupling two phase retrieval problems.  For our first coupled problem, we add a condition involving a fixed reference signal $h$:
$|g-h|=|f-h|.$ We use its geometric interpretation to show that this problem has exactly two solutions. We will also look at the problem with an additional condition involving the Fourier transforms: $|\widehat{g}|=|\widehat{f}|$. This coupled problem is also known as the Pauli problem. Here, we do not get uniqueness for our case, in fact in Theorem \ref{SharpRiesz}, we explicitly construct an uncountable family of solutions using Riesz products. Next, we look at the coupled problem with the condition $|Dg|=|Df|$ where $D$ is a derivation operator. Using the special properties of $D,f$ and $g$, we show that this coupled problem has exactly two solutions. Finally, for our last coupled problem, we add the condition $|g|=|f|$ on a segment on the strip. We show the uniqueness of the solution by using our main results. More precisely, in Theorem \ref{th:anglestrip} we prove that if $f$ and $g$ are holomorphic functions on Hardy space on the strip  and if $|f|=|g|$ on two intersecting segments inside the strip where the angle between these segments is an irrational multiple of $\pi$, then $f=g$ up to unimodular constant.

This work is organized as follows. Section 2 is a quick review of definitions and results on analysis and Hardy spaces. Section 3 is devoted to the solution of the phase retrieval problem in the wide-band case. We will look at the phase retrieval problem on the unit disc and on the strip. Section 4 is devoted to the coupled phase retrieval problems.

\section{Preliminaries}

\subsection{Notation}

For a domain $\Omega\subset\C$, $\mathrm{Hol}(\Omega)$ is the set of holomorphic functions on $\Omega$. For $F\in\mathrm{Hol}(\Omega)$ we denote by
$Z(F)$ the set of zeros of $F$, counted with multiplicity. Write $\overline{\Omega}=\{\bar z:z\in\Omega\}$ and if $F\in\mathrm{Hol}(\Omega)$, we denote by $F^*$ the function in $\mathrm{Hol}(\overline{\Omega})$ defined by $F^*(z)=\overline{F(\bar z)}$. It will be convenient to denote the conjugation function by $C$, where $C(z)=\bar{z}$ for all $z\in\C$. 

The unit disc $\mathbb D$ is defined as $\D=\{z\in\C:|z|<1\}$ and its boundary $\T$ is defined by $\T=\{z\in\C:|z|=1\}$. Let $c>0$ and $\mathcal{S}_c$ be the strip defined as $\mathcal{S}_c:=\{ z \in\mathbb C : |\text{Im} z|<c \}$, and $\mathcal{S}:=\mathcal{S}_1$.

For a nonnegative and locally integrable function $\omega$ on $\mathbb R$, the weighted $L^2$ space on $\R$ is given by
$$
L^2_{\omega}(\mathbb R)=L^2(\mathbb R, \omega \d t)=\Big\{f\text{ is measurable}:||f||^2_{L^2_{\omega}(\mathbb R)}=\int_{\mathbb R}|f(t)|^2\omega(t)\,\d t<+\infty\Big\}.
$$
Finally, consider a measure space $(X_1,\mathcal A_1,\mu)$, a measurable space $(X_2,\mathcal A_2)$, and a measurable map $\psi:X_1\to X_2$. Recall that the pushforward measure of $\mu$ by $\psi$ is given by
$$
{\psi}_*\mu(A) = \mu(\psi^{-1}(A)) 
$$
for $A\in\mathcal A_2$. Equivalently, if $h$ is a function such that $h\circ\psi$ is integrable on $X_1$ with respect to $\mu$, then we have the change of variables formula
$$
\int_{X_2}h\,\d(\psi_*\mu)=\int_{X_1}h\circ\psi\,\d\mu.
$$

\subsection{Hardy Spaces on the Disc}
The Hardy space on the disc $\mathbb D$ is defined as
$$
H^2(\mathbb D)=\Big\{F\in \text{Hol}(\mathbb D): ||F||^2_{H^2(\mathbb D)}=\sup_{0\leq r<1}\dfrac{1}{2\pi}\int_{-\pi}^{\pi}|F(re^{i\theta})|^2~\d\theta<+\infty\Big\}.
$$
We will need the following key facts. First, every $F\in H^2(\D)$ admits a radial limit $F(e^{i\theta})=\lim_{r\to1}F(re^{i\theta})$ for almost every $e^{i\theta}\in \T$
(see e.g. \cite[Theorem 2.2]{Du}) with $F\in L^2(\T)$, $\widehat{F}(n)=0$ for $n=-1,-2,...$, and 
$\log|F|\in L^1(\T)$. Furthermore \cite[Section 7.6]{Ma}, every function $F\in H^2(\D)$ can be uniquely decomposed as
$$
F=e^{i\gamma}B_F S_F O_F
$$
where $e^{i\gamma}\in\T$, $B_F$ is the Blaschke product formed from the zeros of $F$, $S_F$ is a singular inner function, and $O_F$ is the outer part of $F$. More precisely, the Blaschke product is defined as 
\begin{equation}
\label{eq:B-F}
B_F(w)=  \prod_{ \alpha\in Z(F)} b_\alpha(w) , \qquad w\in \D
\end{equation}
where $b_\alpha(w)= \dfrac{\alpha}{|\alpha|}\dfrac{\alpha-w}{1-\bar{\alpha}w}$ and $\displaystyle\sum_{\alpha\in Z(F)}(1-|\alpha|)<\infty$. The singular part is given by
\begin{equation}
\label{eq:S-F}
S_F(w)=\exp\Big(\int_{\mathbb T}\dfrac{w+e^{i\theta}}{w-e^{i\theta}}~\d\nu_F(e^{i\theta})\Big),\qquad w\in \D
\end{equation}
where $\nu_F$ is a finite positive singular measure (with respect to the Lebesgue measure). Finally, the outer part is determined by the modulus of the radial limit of $F$,
\begin{equation}
\label{eq:O-F} O_F(w)=\exp\Big(\dfrac{1}{2\pi}\int_{-\pi}^{\pi}\dfrac{e^{i\theta}+w}{e^{i\theta}-w}\log|F(e^{i\theta})|~\d\theta\Big),\qquad w\in \D.
\end{equation}

\subsection{Hardy Spaces on the Strip}

There are essentially two ways of defining the Hardy space on the strip $\mathcal{S}$.
To start, let us define the conformal bijection $\phi:\mathcal S\longrightarrow \mathbb D$ given by
$$
\phi(z):=\tanh\Big(\frac{\pi}{4}z\Big),\quad z\in\ss.
$$
The inverse mapping is given by the function
\begin{equation}
	\label{inverse-phi}
	\phi^{-1}(w) =\frac{2}{\pi}\ln \frac{1+w}{1-w},\quad w\in \D.
\end{equation}
Observe that $\phi$ has the following properties: $\phi^*=\phi$, $\phi(\R)=\,]-1,1[\,$, 
$\phi:\partial \ss \longrightarrow \T\setminus\{-1,1\}$ is a bijection, and 
$\phi(\partial \mathcal{S}\cap\C_{\pm})=\T^*_{\pm}$, where $\C_+,\C_-$ denote the upper and lower halves of $\C$ 
respectively, and $\T^*_+,\T^*_-$ denote the upper and lower halves of $\T\setminus\{-1,1\}$ respectively.
On one hand we shall consider the following Hardy spaces defined as
$$
H^2(\mathcal S)=\Big\{f\in \text{Hol}(\mathcal S): f\circ\phi^{-1}\in H^2(\mathbb D)\Big\},
$$
and $||f||_{H^2(\mathcal S)}=||f\circ\phi^{-1}||_{H^2(\mathbb D)}$. It can then be shown \cite[Theorem 2.2]{BK} 
that $H^2(\mathcal{S})=H^2_W(\mathcal{S})$  and $||f||_{H^2(\mathcal S)}= ||f||_{H^2_W(\mathcal S)}$  for all  
$f\in H^2(\mathcal{S})$, where
$$
H^2_{W}(\mathcal S)=\Big\{f\in \text{Hol}(\mathcal S): ||f||^2_{H^2_W(\mathcal S)}=\sup_{|y|<1}\int_{\mathbb R}\dfrac{|f(t+iy)|^2}{|W(t+iy)|}~\d t<+\infty\Big\},
$$
and $W(z)=4\cosh^2(\frac{\pi}{4}z)=\dfrac{\pi}{\phi'(z)}$ for all $z\in\mathcal{S}$.

Now this last space can be identified to the natural analogue of the Hardy space on the disc:
$$
H^2_{\tau}(\mathcal S)=\Big\{f\in \text{Hol}(\mathcal S): ||f||^2_{H^2_{\tau}(\mathcal S)}=\sup_{|y|<1}\int_{\mathbb R}|f(t+iy)|^2~\d t<+\infty\Big\}.
$$
More precisely $f\in H^2_\tau (\mathcal S)$ if and only if $W^{1/2}f\in H^2_W(\mathcal S)$, and it was shown in \cite[Chapter VI, Section 7.1]{Ka} that $f\in H^2_\tau (\mathcal S)$ if and only if $\widehat{f}\in L^2(\R,e^{2|\xi|}d\xi)$.

Finally, we obtain the factorization on $H^2_\tau(\mathcal{S})$.

\begin{lemma}
\label{lem:fact-strip}
Let $f\in H^2_{\tau} (\mathcal S)$. Then the unique inner-outer factorization of $f$ is given by
$$
f(z)=e^{i\gamma}\dfrac{B_F(\phi(z))S_F(\phi(z))O_F(\phi(z))}{W(z)^{1/2}},\quad\text{for }z\in\ss
$$
where $F=(W^{1/2})f\circ\phi^{-1}\in H^2(\D)$ and for some $\gamma\in\R$. For all $z\in\mathcal{S}$, the 
Blaschke product $B_f(z)=B_F(\phi(z))$ is given by
\begin{equation}
\label{eq:B-f}
B_f(z)=\prod_{ \beta\in Z(f)}b_{\phi(\beta)}(\phi(z)),
\end{equation}
while the singular inner function $S_f(z)=S_F(\phi(z))$ is given by
\begin{align}
\label{eq:S-f}
\begin{split}
S_f(z)= \exp\left(-a_{\{+1\}} e^{\frac{\pi}{2}z}-a_{\{-1\}}e^{-\frac{\pi}{2}z}+\int_{\mathbb \partial \mathcal S}
\dfrac{\phi(z)+\phi(\zeta)}{\phi(z)-\phi(\zeta)}\,\d\mu_f(\zeta)\right),
\end{split}
\end{align}
for some positive constants $a_{\{\pm 1\}}\geq 0$, where $\mu_f={\phi^{-1}}_*\,\nu_F$ 
is the pushforward measure of $\nu_F$ on $\partial \mathcal S$, and the outer function $O_f(z)=O_F(\phi(z))$ is 
given by
\begin{align}
\label{eq:O-f}
 \begin{split}
O_f(z)&=\exp\Bigg(\dfrac{-1}{2\pi i}\int_{\mathbb R}\dfrac{\phi(z)+\phi(x+i)}{\phi(z)-\phi(x+i)}\dfrac{\phi'(x+i)}{\phi(x+i)}\log|W(x+i)^{1/2}f(x+i)|\,\d x\\
&\quad+\dfrac{1}{2\pi i}\int_{\mathbb R}\dfrac{\phi(z)+\phi(x-i)}{\phi(z)-\phi(x-i)}\dfrac{\phi'(x-i)}{\phi(x-i)}\log|W(x-i)^{1/2}f(x-i)|\,\d x\Bigg).
\end{split}
\end{align}
\end{lemma}
\begin{proof}
For $F\in H^2(\mathbb D)$ and $z\in \mathcal S$, according to the above connection between Hardy spaces, we have 
$F(\phi(z))=W^{1/2}(z)f(z)$ and equivalently,
$$
f(z)=\dfrac{F(\phi(z))}{W(z)^{1/2}}=e^{i\gamma}\dfrac{B_F(\phi(z))S_F(\phi(z))O_F(\phi(z))}{W(z)^{1/2}}.
$$
Note that this is well-defined on $\mathcal S$ since $W(z)=\dfrac{\pi}{\phi'(z)}\neq 0$ for any $z\in \mathcal S$. 
	
The formula for the Blaschke product easily follows from $B_f(z)=B_F(\phi(z))$. 
For the singular inner part, since $\phi:\partial \ss\to \T\setminus\{-1,1\}$ is a bijection, we have for 
$z\in \mathcal{S}$
\begin{align*}
S_f(z)&=S_F(\phi(z))\\
&=\exp\left(\nu_F(\{1\}) \frac{\phi(z)+1}{\phi(z)-1} + \nu_F(\{-1\})\frac{\phi(z)-1}{\phi(z)+1}  
+ \int_{\mathbb T\setminus\{1,-1\}}\dfrac{\phi(z)+e^{i\theta}}{\phi(z)-e^{i\theta}}~\d\nu_F(e^{i\theta})\right)\\
&=\exp\left(-\nu_F(\{1\}) e^{\frac{\pi}{2}z}-\nu_F(\{-1\}) e^{-\frac{\pi}{2}z}  
+ \int_{\mathbb \partial \mathcal S}\dfrac{\phi(z)+\phi(\zeta)}{\phi(z)-\phi(\zeta)}\,\d\mu_f(\zeta) \right),
\end{align*}	
where  $\mu_f={\phi^{-1}}_*\,\nu_F$ is the pushforward measure of $\nu_F$ on $\partial \mathcal S$. Since $\nu_F$ is 
a positive singular measure, $\nu_F(\{\pm1\})=0$ if $\nu_F$ has no mass at $1$, otherwise $\nu_F(\{\pm1\})>0$.
For the outer function, we also need to split the integral since $\phi(\partial\mathcal{S}\cap\C_\pm)=\T^*_\pm$. 
Hence, the outer function is given by
\begin{align*}
O_f(z)=O_F(\phi(z))&=\exp\Bigg(-\dfrac{1}{2\pi}\int_{0}^{\pi}\dfrac{\phi(z)+e^{i\theta}}{\phi(z)-e^{i\theta}}
\log|F(e^{i\theta})|\,\d\theta\\
&\qquad+\dfrac{1}{2\pi}\int_{-\pi}^0\dfrac{\phi(z)+e^{i\theta}}{\phi(z)-e^{i\theta}}\log|F(e^{i\theta})|
\,\d\theta\Bigg)
\end{align*}
for all $z\in\mathcal S$. By applying the substitutions $e^{i\theta}=\phi(x+i),~\theta\in\,]0,\pi[\,$ on 
$\mathbb T^*_+$ and $e^{i\theta}=\phi(x-i),~\theta\in\,]-\pi,0[\,$ on $\mathbb T^*_-$, we get \eqref{eq:O-f}.
\end{proof}

\begin{remark} Since $W^{1/2}(z)=1/(2\cosh(\frac{\pi}{4}z))$ for $z\in \mathcal{S}$, by $\eqref{inverse-phi}$, we get  
$W^{-1/2}(\phi^{-1}(w))={\sqrt{1-w^2}}/{2}, \quad w\in \D$, and so   $W^{-1/2}\circ\phi^{-1}$ is a bounded outer 
function on the disc. Let $f(z)=\exp\left(-e^{\frac{\pi}{2}z}\right)W^{-1/2}(z)\in H^2_\tau(\mathcal{S})$, we have
$((W^{1/2}f)\circ \phi^{-1})(z)= \exp\left(\frac{z+1}{z-1}\right)$, $z\in \D$ and so $(W^{1/2}f)\circ \phi^{-1}$ is 
the singular inner function in $H^2(\D)$ associated with the Dirac measure at $1$.
\end{remark}

\section{Phase Retrieval in $H_\tau^2(\mathcal{S})$}

\subsection{Reduction of the Problem}
In this section, we consider $f,g\in L^2(\mathbb R)$ with $\widehat{f},\widehat{g}\in L^2(\mathbb R, e^{2c|\xi|}\d\xi)$ such that 
$|f(x)|=|g(x)|$ for every $x\in\R$. Our goal is to determine, for a given $f$, all possible $g$'s.

To do so, let us write $f_c(x)=f(cx)$ and $g_c(x)=g(cx)$ so that $f_c,g_c\in L^2(\mathbb R)$ with $\widehat{f_c},\widehat{g_c}\in L^2(\mathbb R, e^{2|\xi|}\d\xi)$ and $|f_c(x)|=|g_c(x)|$ for every $x\in\R$ so that it is enough to consider the case $c=1$.

Note that $\widehat{f},\widehat{g}\in L^2(\R, e^{2|\xi|}\d\xi)$ if and only if $f,g\in H^2_\tau(\mathcal{S})$. Thus, $f$ and $g$ extend holomorphically to $\mathcal{S}$ and $|f(x)|=|g(x)|$ for every $x\in\R$ can be written as
\begin{equation}
	\label{eq:phase}
	f(x)\overline{f(\bar x)}=g(x)\overline{g(\bar x)},\qquad x\in \R.
\end{equation}
But now, \eqref{eq:phase} is an equality between two holomorphic functions on $\R$ so that it is valid also for all $x\in\mathcal{S}$. In other words, we are now trying to solve the following problem: given $f\in H^2_\tau(\mathcal{S})$, find all $g\in H^2_\tau(\mathcal{S})$ such that
\begin{equation}
	\label{eq:h2tau}
	f(z)f^*(z)=g(z)g^*(z),\qquad z\in \mathcal{S}.
\end{equation}

It turns out that this problem is easier to solve when transfering the problem to the disc. Multiplying by $W^{1/2}(z)$, $\overline{W^{1/2}(\bar z)}$ both sides of \eqref{eq:h2tau},
we obtain
$$
(W^{1/2}f)(z)\overline{(W^{1/2}f)(\bar{z})}=(W^{1/2}g)(z)\overline{(W^{1/2}g)(\bar{z})}
$$ 
for all $z\in\mathcal{S}$.
Observe that the functions $F=W^{1/2}f\circ \phi^{-1}$ and $G=W^{1/2}g\circ\phi^{-1}$ are in $H^2(\mathbb D)$. Hence, by applying the substitution $z=\phi^{-1}(w)$ and $\bar{z}=\phi^{-1}(\bar{w})$ to the previous equation, we get
\begin{equation}
\label{eq:h2d}
F(w)F^*(w)=G(w)G^*(w) ,\qquad w\in\D. 
\end{equation}
Therefore, we have translated the equality on the strip to an equivalent equality on the disk. Finally, we are now trying to solve the following problem on the disc: given $F\in H^2(\D)$, find all $G\in H^2(\D)$ such that \eqref{eq:h2d} holds for all $w\in\D$. Note that \eqref{eq:h2d} is equivalent to
$|F(w)|^2=|G(w)|^2$ for $w\in(-1,1)$.

\subsection{The Phase Retrieval Problem on the Disc}

In this section, we look at the equivalent phase retrieval problem on the disc.

Let $F\in H^2(\D)$ and write $F=B_FS_FO_F$ with $B_F,S_F,O_F$ given in equations \eqref{eq:B-F}, \eqref{eq:S-F} and \eqref{eq:O-F}, respectively. The factorization of $F^*$ is given by

$$
F^*=e^{i\lambda}B_{F^*}S_{F^*}O_{F^*}=e^{i\lambda}B_F^*S_F^*O_F^*.
$$

Since the factorization in $H^2(\D)$ is unique, we have $B_{F^*}=B_F^*,~S_{F^*}=S_F^*$, and $O_{F^*}=O_F^*$.
Hence, for all $w\in\D$, the Blaschke product formed from the zeros of $F^*$ is given by
\begin{equation}
	\label{eq:B-Fstar}
	 B_{F^*}(w)=B_F^*(w)=\prod_{\alpha\in Z(F)}b_{\bar{\alpha}}(w)=\prod_{\alpha\in \overline{Z(F)}}b_{\alpha}(w).
\end{equation}
The singular part of $F^*$ is given by
\begin{equation}
	\label{eq:S-Fstar}
	S_{F^*}(w)=S_F^*(w)=\exp\Big(\int_{\mathbb T}\dfrac{w+e^{i\theta}}{w-e^{i\theta}}~\d(C_*\nu_F)(e^{i\theta})\Big),
\end{equation}
for all $w\in\D$, where $C_*\nu_F$ is the pushforward measure of $\T$ by the conjugation function $C$. Finally, for all $w\in\D$, the outer part of $F^*$ is given by
\begin{equation}
	\label{eq:O-Fstar}
	O_{F^*}(w)=O_F^*(w)=\exp\Big(\dfrac{1}{2\pi}\int_{-\pi}^{\pi}\dfrac{e^{i\theta}+w}{e^{i\theta}-w}\log|F(e^{-i\theta})|~\d\theta\Big).
	\end{equation}
We use all of the facts above to prove the following lemma.
\begin{lemma}
	\label{lem:disc}
	Let $F, G\in H^2(\mathbb D)$. Then 
	\begin{equation*}
		|F(w)|^2=|G(w)|^2,\qquad  w\in (-1,1)
	\end{equation*}
	if and only if
	\begin{enumerate}
	\renewcommand{\theenumi}{\roman{enumi}}
		\item the zero sets of $F$ and $G$ satisfy
		$$
		Z(F)\cup \overline{Z(F)}=Z(G)\cup \overline{Z(G)};
		$$
		\item the singular measures $\nu_F$ and $\nu_G$, associated with $F$ and $G$ respectively, satisfy
		$$
		\nu_F+C_*\nu_F=\nu_G+C_*\nu_G \qquad\text {on } \T;
		$$
		\item the radial limits satisfy
		$$
		|F(e^{i\theta})F(e^{-i\theta})|=|G(e^{i\theta})G(e^{-i\theta}) | \qquad \text{ a.e. on }\mathbb T.$$
	\end{enumerate}
\end{lemma}
\begin{proof}
	Let $F, G\in H^2(\mathbb D)$. Note that $FF^*$ and $GG^*$ have decompositions given by
	
	$$
	FF^*=B_FB_{F^*}S_FS_{F^*}O_FO_{F^*}\qquad \text{ and }\qquad\, GG^*=B_GB_{G^*}S_GS_{G^*}O_GO_{G^*}.
	$$
	
	Notice that $B_FB_{F^*}$ is again a Blaschke product, $S_FS_{F^*}$ is again a singular inner function, and $O_FO_{F^*}$ is again an outer function. Indeed, for all $w\in\mathbb D$, \eqref{eq:B-Fstar} implies that
	
	$$
	B_F(w)B_{F^*}(w)=\prod_{ \alpha\in Z(F)\cup \overline{Z(F)}}b_\alpha(w),
	$$
	while \eqref{eq:S-Fstar} implies that
	$$
	S_F(w)S_{F^*}(w)=\exp\Big(\dfrac{1}{2\pi}\int_{\mathbb T}\dfrac{w+e^{i\theta}}{w-e^{i\theta}}~\d\Big(\nu_F+C_*\nu_F\Big)(e^{i\theta})\Big),
	$$
	and finally,  \eqref{eq:O-Fstar} implies that
	$$
	O_F(w)O_{F^*}(w)=\exp\Big(\dfrac{1}{2\pi}\int_{-\pi}^{\pi}\dfrac{e^{i\theta}+w}{e^{i\theta}-w}\log|F(e^{i\theta})F(e^{-i\theta})|\,\d\theta\Big).
	$$
	\par Thus, writing the same for $GG^*$ and using the uniqueness of the decomposition,
	$FF^*=GG^*$
	implies that $B_FB_{F^*}=B_GB_{G^*}$, which in turn implies that
	$$
	Z(F)\cup\overline{Z(F)}=Z(G)\cup\overline{Z(G)}.
	$$
	Furthermore, $FF^*=GG^*$ also implies that
	$$
	S_FS_{F^*}=S_GS_{G^*}\qquad\text{ and }\qquad
	O_FO_{F^*}=O_GO_{G^*}.
	$$ 
	Thus,  
	$$
	\nu_F+C_*\nu_F=\nu_G+C_*\nu_G
	$$
	on $\T$, and by Fatou's theorem \cite[Theorem 2.2]{Du}, we have for almost every $\theta\in\R$
	\begin{align*}
	\lim_{r\rightarrow 1}(O_F(re^{i\theta})O_{F^*}(re^{i\theta}))&=\lim_{r\rightarrow 1}(O_G(re^{i\theta})O_{G^*}(re^{i\theta})),\\
	\intertext{which in turn implies that}
	|F(e^{i\theta})F(e^{-i\theta})|&=|G(e^{i\theta})G(e^{-i\theta})|
	\end{align*}
	almost everywhere on $\mathbb T$.
\end{proof}
We can now construct such $G$'s to solve the equivalent phase retrieval problem on the disc. 
Let $\mathcal{N}^+$ denote the Smirnov class, namely those functions holomorphic on $\D$  of the form 
$f=g/h$, where $g,h$ are bounded analytic functions on $\D$ 
 such that  $h$ is an outer function. If 
$g$ is outer, then $f$ is also outer.
 Note that if $f\in \mathcal N^+$, then by Fatou's Theorem \cite[Theorem 1.3]{Du}, the radial limit $f_*$  exists almost everywhere on $\mathbb T$ and $\log |f_*|\in L^1(\T)$.
The following corollary immediately follows from Lemma \ref{lem:disc}.
\begin{corollary}
	\label{cor:discsoln}
	Let $F,G\in H^2(\D)$. Then $|F|=|G|$ on $(-1,1)$ if and only if the inner-outer decomposition of $F$ and $G$ are given by
	$$
	F=e^{i\gamma}B_FS_FO_F\qquad \text{ and }\qquad G=e^{i\gamma'}B_GS_GO_G
	$$
	where
	\begin{enumerate}
		\item $B_F, S_F, O_F$ are given by \eqref{eq:B-F}, \eqref{eq:S-F}, \eqref{eq:O-F} respectively;
		\item $B_G$ is the Blaschke product associated with the set $A\cup(\overline{Z(F)\backslash A})$ for some $A\subset Z(F)$;
		\item $S_G$ is the singular inner function associated with the positive singular measure $\nu_G=\nu_F+\rho$, where $\rho$ is an odd real singular measure; and
		\item $O_G=UO_F$ where $U\in \mathcal N^+$ is an outer function and $U=1/U^*$ on $\D$.
	\end{enumerate}
\end{corollary}

\begin{remark} We can make the condition on $\rho$ a bit more explicit so as to be constructive. We write $\rho=\rho_+-\rho_-$, where $\rho_+$ is the positive part while $\rho_-$ is the negative part. In particular, $\rho_+$ and $\rho_-$ have disjoint supports. The fact that $\rho$ is an odd measure, $C_*\rho=-\rho$  is equivalent to $\rho_-=C_*\rho_+$. The fact that $\nu_G$ is positive, is equivalent to the condition $C_*\rho_+\leq \nu_F$, or equivalently, $\rho_+\leq C_*\nu_F$. Conversely, take a set $E\subset\T$ such that $E\cap\overline{E}=\emptyset$ and a positive singular measure $\rho_+$ supported in $E$ and such that $\rho_+\leq C_*\nu_F$. Then we can take $\rho=\rho_+-C_*\rho_+$.
\end{remark}

\begin{proof}[Proof of Corollary \ref{cor:discsoln}]
	Let $F,G\in H^2(\D)$ with inner-outer decompositions as defined on Corollary \ref{cor:discsoln}. Observe that the properties of the Blaschke product $B_G$ and the singular inner function $S_G$ immediately follow from Lemma \ref{lem:disc}. For the outer function, by Lemma \ref{lem:disc}, we have
	\begin{equation}
	\label{eq:disc-fac}
	|O_F(e^{i\theta})O_F(e^{-i\theta})|=|O_G(e^{i\theta})O_G(e^{-i\theta})|
	\end{equation}
	almost everywhere on $\T$. Hence, 
	$$
	\log|O_G(e^{i\theta})|=\log|O_F(e^{i\theta})|+\log|U(e^{i\theta})|
	$$
	almost everywhere on $\T$, where $\log|U(e^{i\theta})|$ is an odd real-valued function of $\theta$ and $\log|U|\in L^1(\T)$.   Since $|O_G(e^{i\theta})|=|O_F(e^{i\theta})U(e^{i\theta})|$ almost everywhere on $\T$ and $O_G$ and $O_F$ are outer functions, we get 
	$$O_G(z)=O_F(z)O_U(z),\qquad z\in \D.$$
	Hence $ O_U=O_G/O_F\in \mathcal N^+$.  
	  Moreover, \eqref{eq:disc-fac} implies that 
	$|O_U(e^{i\theta})O_U(e^{-i\theta})|=1$ almost everywhere on $\T$, and so $O_U(z)O_U^*(z)=1$ on $\D$.
	\end{proof}
	
We can actually identify the solutions of the phase retrieval problem on the disc in terms of a 
factorization. Let us consider an analog of the result of McDonald \cite[Proposition 1]{Mc}.

\begin{corollary}
	\label{cor:uv-disc}
	Let $F,G\in H^2(\D)$. Then $|F|=|G|$ on $(-1,1)$ if and only if there exist $u,v\in \mathrm{Hol}(\D)$ such that $F=uv$ and $G=uv^*$.
\end{corollary}
\begin{proof}
	Let $F,G\in H^2(\D)$. By Corollary \ref{cor:discsoln}, we have the factorizations
	$$
	F=B_FS_FO_F\qquad \text{ and }\qquad G=B_FS_FO_F.
	$$
	First, observe that we can write the Blaschke products $B_F$ and $B_G$ as
	$$
	B_F=B_1B_2\qquad \text{ and }\qquad B_G=B_1B_2^*
	$$
	where
	$B_1$ is the Blaschke product associated with $A\subset Z(F)$ and $B_2$ is the Blaschke product associated with $Z(F)\backslash A$. On the other hand, we can write the singular measures $\nu_F$ and $\nu_G$ as 
	$$
	\nu_F=\nu_1+\nu_2\qquad \text{ and }\qquad\nu_G=\nu_1+C_*\nu_2
	$$
	where 
	$$\nu_1=\nu_F+\dfrac{\rho_+-C_*\rho_+}{2}\qquad \text{ and }\qquad \nu_2=\dfrac{C_*\rho_+-\rho_+}{2},
	$$
	so that $S_F=S_{\nu_1}S_{\nu_2}$ and $S_G=S_{\nu_1}S_{\nu_2}^*$. 
	
	Since $O_G=UO_F$  where $U$ is an outer function, we have $U\in \mathcal N^+$ and $UU^*=1$ on $\D$. We write
	$$
	O_F=O_FU^{1/2}U^{-1/2}
	$$
	and
	$$
	O_G=O_FU^{1/2}U^{1/2}=O_FU^{1/2}(U^{-1/2})^*
	$$
	Therefore, we take 
	$$
	u=B_1S_{\nu_1}O_FU^{1/2}\qquad \text{ and }\qquad v=B_2S_{\nu_2}U^{-1/2}.
	$$
\end{proof}

\subsection{Back to the Strip}
In this section, we go back to the phase retrieval problem on the strip. Using Lemma \ref{lem:fact-strip}, we see that Lemma \ref{lem:disc} translates to functions on $H^2_{\tau}(\mathcal S)$. By a change of variable and by applying the inner-outer factorization on $H^2_\tau(\mathcal{S})$, we have:
\begin{lemma} 
\label{lem:strip}	
Let $f,g\in H^2_{\tau}(\mathcal S)$. Then 
\begin{equation*}
|f(z)|^2=|g(z)|^2,\qquad z\in\R
\end{equation*}
if and only if
\begin{enumerate}
\renewcommand{\theenumi}{\roman{enumi}}
	\item the zero sets of $f$ and $g$ satisfy
	$$
	Z(f)\cup \overline{Z(f)}=Z(g)\cup \overline{Z(g)};
	$$
	\item the singular measures $\mu_f$ and $\mu_g$, associated with $f$ and $g$ respectively, satisfy
	$$
	\mu_f+C_*\mu_f=\mu_g+C_*\mu_g \qquad \text { on }\partial\mathcal S;
	$$ 
	\item the boundary values satisfy
	$$
	|f(x+i)f(x-i)|=|g(x+i)g(x-i)| \qquad \text{ a.e. on } \R.$$
\end{enumerate}
\end{lemma}

We now construct the solutions of the problem on the strip.  
Let $\mathcal N^+_\tau(\mathcal{S})$ be the Smirnov class of holomorphic functions in $\mathcal{S}$ such that $f=(F\circ\phi)/W^{1/2}$ where $F\in \mathcal N^+$. The following result immediately follows from Lemma \ref{lem:strip}.
\begin{theorem}
	\label{thm:stripsoln}
	Let $f,g\in H^2_\tau(\mathcal{S})$. Then $|f|=|g|$ on $\R$ if and only if the inner-outer decomposition of $f$ and $g$ are given by
	$$
	f=e^{i\gamma}W^{-1/2}B_fS_fO_f \qquad \text{ and }\qquad g=e^{i\gamma'}W^{-1/2}B_gS_gO_g
	$$
	 where
	\begin{enumerate}
		\item $B_f,S_f,O_f$ are given by \eqref{eq:B-f}, \eqref{eq:S-f}, \eqref{eq:O-f} respectively;
		\item $B_g$ is the Blaschke product associated with the set $A\cup(\overline{Z(f)\backslash A})$ with $A\subset Z(f)$;
		\item $S_g$ is the singular inner function associated with the positive singular measure $\mu_g=\mu_f+\sigma$, where $\sigma$ is an odd real singular measure, given by $\sigma=\sigma_+-C_*\sigma_+$, satisfying $C_*\sigma=-\sigma$ and $0\leq\sigma_+\leq C_*\mu_f$; and
		\item $O_g$ is the outer part of $uO_f$ where $u\in \mathcal N^+_\tau(\mathcal{S})$ is an outer  function  and $u=1/u^*$  on $\mathcal{S}$.
	\end{enumerate}
\end{theorem}

\begin{remark}
	Observe that possible trivial solutions to the problem on the strip are given by:
$$
(1)\quad g(z)=ce^{i\eta z}f(z)\qquad\mbox{and}\qquad
(2)\quad g(z)=ce^{i\eta z}f^*(z)
$$
with $|c|=1$ and $\eta\in\R$. These trivial solutions are retrieved as follows:

\begin{enumerate}

\item the factor $e^{i\eta z}$ is the factor $u$ of the outer part as $e^{i\eta z}(e^{i\eta z})^*=1$

\item the replacement of $f$ by $f^*$ is obtained by taking $A=\emptyset$ for the Blaschke part, $\sigma=C_*\mu_f-\mu_f$ so that $\mu_g=C_*\mu_f$ for the inner part and finally $u=O_{f^*}/O_f$ so that the outer part of $g$ is $O_g=uO_f=O_{f^*}$.
\end{enumerate}
\end{remark}

Corollary \ref{cor:uv-disc} also translates to a result on the strip with a simple change of variable.
\begin{corollary}
	\label{cor:uv-strip}
	Let $f,g\in H^2_\tau(\mathcal{S})$. Then $|f|=|g|$ on $\R$ if and only if there exist $u,v\in \mathrm{Hol}(\mathcal{S})$ such that $f=uv$ and $g=uv^*$.
\end{corollary}

Finally, we go back to our initial phase retrieval problem. The following result directly follows from Theorem \ref{thm:stripsoln}.
\begin{corollary}\label{maincoro}
	Let $f\in L^2(\mathbb R)$ and $\widehat{f}\in L^2(\mathbb R, e^{2c|\xi|}d\xi)$. For all $g\in L^2(\mathbb R)$ such that $\widehat{g}\in L^2(\mathbb R, e^{2c|\xi|}d\xi)$ with $|f(x)|=|g(x)|$ for all $x\in\R$, $g$ can be written as
	$$
	g=e^{i\kappa}W^{-1/2}B_gS_gO_g
	$$
	where
	\begin{enumerate}
		\item $B_g$ is the Blaschke product associated with the set $A\cup(\overline{Z(f)\backslash A})$ with $A\subset Z(f)$;
		\item $S_g$ is the singular inner function associated with the positive singular measure $\mu_g=\mu_f+\sigma$, where $\sigma$ is an odd real singular measure, given by $\sigma=\sigma_+-C_*\sigma_+$, satisfying $C_*\sigma=-\sigma$ and $0\leq\sigma_+\leq C_*\mu_f$; and
		\item  $O_g$ is the outer part of $uO_f$ where $u\in \mathcal N^+_\tau(\mathcal{S})$ is an outer  function and $u=1/u^*$  on $\mathcal{S}$.
	\end{enumerate}
\end{corollary}

\section{Coupled Phase Retrieval Problems}
In this section, we are investigating coupled phase retieval problems, i.e. problems of the form $|u|=|v|$, $|Tu|=|Tv|$ where $T$ is some transform. This additional assumption involving $T$ may either lead to uniqueness or at least to the reduction of the set of solutions.

\subsection{Adding a Fixed Reference}
Klibanov {\it et. al.} \cite{Kli} considered the following constrained problem in the band-limited case: 
\begin{equation}
	\label{eq:klibanov}
	|g|=|f| \quad \text{and}\quad |g-h|=|f-h|
\end{equation}
where $h$ is a suitable fixed reference signal. They were able to show that there are at most two solutions of this problem. For the following result, we look at a similar problem. It turns out that for the wide-band case, we also obtain two solutions. 

\begin{theorem}
	Let $f,g\in H^2_{\tau}(\mathcal{S})$ and $h$ be a nonzero complex-valued function such that $\Phi=e^{i\arg h}$ is analytic on $\R$. Suppose that $|g(x)|=|f(x)|$ and $|g(x)-h(x)|=|f(x)-h(x)|$ for (a.e.) $x\in\R$. Then there exists two solutions of this problem, namely $g(x)=f(x)$ or $g(x)=\overline{f(x)}\Phi(x)^2$, for $x\in\R$.
\end{theorem}

\begin{proof}
	Consider the two circles in $\C$: $\mathcal C(0,|f(x)|)$ and $\mathcal C(h(x),|f(x)-h(x)|)$. These two circles have two intersection points, one being $f(x)$, the other being $\overline{f(x)}\Phi(x)^2$ (eventually being the same as the first one). 
	\begin{center}
		\begin{tikzpicture}[line cap=round,line join=round,>=triangle 45,x=2.361917598474852cm,y=2.3619725559168456cm]
	\draw[color=black] (-0.8975623766832053,0.) -- (2.0712118110902464,0.);
	\foreach \x in {-0.5,0.5,1.,1.5,2.}
	\draw[shift={(\x,0)},color=black] (0pt,-2pt);
	\draw[color=black] (0.,-0.6756260552541451) -- (0.,1.7291436300545806);
	\foreach \y in {-0.5,0.5,1.,1.5}
	\draw[shift={(0,\y)},color=black] (-2pt,0pt);
	\clip(-0.8975623766832053,-0.6756260552541451) rectangle (2.0712118110902464,1.7291436300545806);
	\draw [line width=1.pt,color=blue] (0.,0.) ellipse (1.3824866751867755cm and 1.3825188430876827cm);
	\draw [line width=1.pt,color=blue] (1.0476687316797637,0.7552908529911855) ellipse (2.0936673898810967cm and 2.093716105638246cm);
	\draw [line width=1.pt,dash pattern=on 4pt off 4pt] (0.,0.)-- (0.18401695336559143,0.5556453517332945);
	\draw [line width=1.pt,dash pattern=on 4pt off 4pt] (0.,0.)-- (0.5853229693223431,-0.0010087677861107705);
	\draw [line width=1.pt,dash pattern=on 4pt off 4pt] (0.,0.)-- (1.0476687316797637,0.7552908529911855);
	\draw [line width=1.pt,dash pattern=on 4pt off 4pt] (0.18401695336559143,0.5556453517332945)-- (1.0476687316797637,0.7552908529911855);
	\draw [line width=1.pt,dash pattern=on 4pt off 4pt] (1.0476687316797637,0.7552908529911855)-- (0.5853229693223431,-0.0010087677861107705);
	\draw [line width=1.pt,dash pattern=on 4pt off 4pt] (0.18401695336559143,0.5556453517332945)-- (0.5853229693223431,-0.0010087677861107705);
	\draw (0.06310567990319361,0) node[anchor=north west] {${}_{|f(x)|}$};
	\draw (-0.14762150670285515,0.2602508068118796) node[anchor=north west] {${}_0$};
	\draw (0.8998165678977991,1.00659726085905763) node[anchor=north west] {${}_{h(x)}$};
	\draw (0.8578379836019023,0.5143630673728532) node[anchor=north west] {${}_{|f(x)-h(x)|}$};
	\draw [color=ffqqqq](0.8271107969958536,-0.32359939414588021) node[anchor=north west] {${}_{f(x)}$};
	\draw [color=ffqqqq](-0.7983966418097705,1.0659726085905763) node[anchor=north west] {${}_{\overline{f(x)}\Phi(x)^2}$};
	\draw [->,line width=0.8pt] (-0.16663221455490681,0.8618603480296028) -- (0.1606038985191835,0.5850847002079304);
	\draw [->,line width=0.8pt] (0.8762765318692888,-0.37938013426902037) -- (0.5992614757199941,-0.027990725492131302);
	\begin{scriptsize}
	\draw [fill=black] (0.,0.) circle (2.0pt);
	\draw [fill=black] (1.0476687316797637,0.7552908529911855) circle (2.0pt);
	\draw [fill=ffqqqq] (0.18401695336559143,0.5556453517332945) circle (2.0pt);
	\draw [fill=ffqqqq] (0.5853229693223431,-0.0010087677861107705) circle (2.0pt);
	\end{scriptsize}
	\end{tikzpicture}

\smallskip

The circles $\mathcal C(0,|f(x)|)$ and $\mathcal C(h(x),|f(x)-h(x)|)$.
	\end{center}
	Therefore, for each $x\in\R$, either $g(x)=f(x)$ or $g(x)=\overline{f(x)}\Phi(x)^2$. By the pigeonhole principle, one of these two alternatives is valid on a set of positive measure. But $f,g$ and $\bar{f}\Phi^2$ are all analytic so that if $g=f$ on a set of positive measure, then $g=f$ everywhere, otherwise if $g=\bar{f}\Phi^2$ on a set of positive measure, then $g=\bar{f}\Phi^2$ everywhere as well.
\end{proof}

\begin{remark}
If we do not assume $\Phi$ to be analytic, then $\bar f\Phi^2$ may not be analytic and would therefore not be a solution.
\end{remark}

\subsection{Pauli's Problem}
For our next result, we add a constraint involving the Fourier transforms:
\begin{equation}
\label{eq:pauli}
|g|=|f| \qquad \text{ and } \qquad |\widehat{g}|=|\widehat{f}|
\end{equation}
This problem is due to Pauli, who speculated that \eqref{eq:pauli} would imply $g=cf$ for some $c\in\T$.
However, one may construct many pairs $(f,g)$ satisfying \eqref{eq:pauli}
for which this is not the case (see e.g. Vogt \cite{Vo}, Corbett and Hurst \cite{CH,CH1}).
Such pairs are now called \textit{Pauli partners}. In the band-limited case, Ismagilov \cite{Is} and the first author \cite{Ja} have independently shown that the set of the Pauli partners may be arbitrarily large. However, although this is not explicitly stated in \cite{Is,Ja}, for
a given band-limited $f$ only finitely band-limited partners (up to trivial solutions) are constructed. The following result shows that the solution set of the Pauli problem in the wide-band case may be arbitrarily large as well and even uncountable.

\begin{theorem}
	\label{SharpRiesz}
	There exists $f\in H^2_{\tau}(\mathcal{S})$ which has a nondenumerable infinity of Pauli partners which are not constant multiples of one another.
\end{theorem}

\begin{proof} The proof is a direct adaptation of \cite{Is,Ja}.
	
	Let $\{\alpha_n\}_{n=0}^{\infty}$ be a sequence of non-zero real numbers such that $\displaystyle\sum_{n=1}^{+\infty}|\alpha_n|^2<\infty$ and consider the associated Riesz product 
	$$
	R_\alpha(x)=\prod_{n=1}^{\infty}\big(1+2i\alpha_n\sin(2\pi3^nx)\big).
	$$
	For properties of Riesz products, we refer the reader to the book of Katznelson \cite{Ka}.
 	We may write this Riesz product as a Fourier series 
	 \begin{equation}
		\label{eq:riesz}
		R_\alpha(x)=\displaystyle\sum_{k\in\mathbb Z}a_ke^{2\pi i kx}.
	 \end{equation}
	 
	Next, let $\varphi\in L^2(\R)$ be such that $\widehat{\varphi}$ is supported on $[0,1]$ and bounded. For all $x\in\R$, take $f=R_\alpha\varphi$. As
	$$
	f(x)=
	\Big(\sum_{k\in\mathbb Z}a_ke^{2\pi i kx}\Big)\varphi(x),
	$$
	we get
	$$
	\widehat{f}(\xi)=\sum_{k\in\mathbb Z}a_k\widehat{\varphi}(\xi-k).
	$$

	Now, observe that $a_k=0$ unless there exists an integer $N$ and $\eta_1,\ldots,\eta_N\in\{-1,0,1\}$ with $\eta_N\not=0$ such that $\displaystyle k=\sum_{j=1}^N\eta_j3^j$. Further, $N$ and the $\eta_j$'s are uniquely determined by $k$. In this case, a simple computation shows that $3^{N-1}\leq|k|\leq 3^{N+1}$ and that
	\begin{equation}
		\label{eq:rieszFourier}
		|a_k|=\prod_{j=1,\,\eta_j\ne 0}^N|\alpha_j|.
	\end{equation}
	Therefore, if we choose $0<|\alpha_j|\leq e^{-2\cdot3^{j+1}}$, we get
	$$
	|a_k|\leq |\alpha_N|\leq  e^{-2\cdot 3^{N+1}}\leq e^{-2 |k|}.
	$$
	As a consequence, for $k\leq |\xi|\leq k+1$,
	$$
	|\widehat{f}(\xi)|=|a_k||\widehat{\varphi}(\xi-k)|\leq e^{-2|k|}\|\widehat{\varphi}\|_\infty\leq Ce^{-2|\xi|}.
	$$
	It follows that $f\in H^2_\tau(\mathcal{S})$.

	Next, let $\varepsilon=\{\varepsilon_n\}_{n=1}^{\infty}\in\{-1,1\}^{\mathbb N}$ and
	$\alpha(\varepsilon)=\{\alpha_n\varepsilon_n\}_{n=1}^{\infty}$. In particular, for $\varepsilon=\mathbf{1}=(1,1,\ldots)$, $\alpha(\mathbf{1})=\alpha$.
	Observe that the associated Riesz product 
	$$
	R_{\alpha(\varepsilon)}(x)=\prod_{n=1}^{\infty}\big(1+2i\alpha_n\varepsilon_n\sin(2\pi3^nx)\big)
	 =\sum_{k\in\mathbb Z}a_k(\varepsilon)e^{2\pi i kx}
	$$
	has the following properties:
	\begin{itemize}
		\item for every $x\in\R$, $|R_{\alpha(\varepsilon)}(x)|=|R_\alpha(x)|$;
		\item for every $k\in\mathbb Z$, $|a_k(\varepsilon)|=|a_k|$.
	\end{itemize}
	This last property follows directly from \eqref{eq:rieszFourier}. Note also that $R_{\alpha(\varepsilon)}$ is not a constant multiple of $R_{\alpha(\varepsilon')}$ if $\varepsilon\neq\varepsilon'$.

	It remains to define $f_\varepsilon=R_{\alpha(\varepsilon)}\varphi$. Then $f_\varepsilon$ has the following properties:
	\begin{itemize}
		\item $f_\varepsilon\in H^2_\tau(\mathcal{S})$ and $f_\varepsilon$ is not a constant multiple of $f_{\varepsilon'}$ if $\varepsilon\neq\varepsilon'$;
		\item $|f_\varepsilon(x)|=|f_{\varepsilon'}(x)|$ for all $x\in\R$;
		\item $|\widehat{f_\varepsilon}(\xi)|=|\widehat{f_{\varepsilon'}}(\xi)|$ since for $k\leq|\xi|\leq k+1$, $k\in\mathbb Z$, 
		$$|\widehat{f_\varepsilon}(\xi)|=|a_k(\varepsilon)||\widehat{\varphi}(\xi-k)|=|a_k(\varepsilon')||\widehat{\varphi}(\xi-k)|=|\widehat{f_{\varepsilon'}}(\xi)|.$$
	\end{itemize}

\end{proof}

\subsection{Derivation Operator}
We now look at a direct consequence of Corollary \ref{cor:uv-strip}. Let $b,q\in\R$ with $|q|<1$. For all $z\in \mathcal{S}$ and $f\in H^2_{\tau}(\mathcal{S})$, consider the operator $\dfrac{\partial}{\partial z}$ where $\dfrac{\partial}{\partial z}f(z)=f'(z)$, the operator $\delta$ given by
$$
\delta(f)(z)=f(z+b)-f(z),
$$
and the operator $\gamma$ given by
$$
\gamma(f)(z)=f(qz)-f(z).
$$
Let $D$ be one of $\dfrac{\partial}{\partial z},\delta$ or $\gamma$. McDonald \cite[Theorem 1]{Mc} considered the coupled phase retrieval problem: $f,g$ entire, $|g(x)|=|f(x)|$ with the additional constraint $|Dg(x)|=|Df(x)|$ for $x\in\R$. McDonald showed that if $f=uv$ and $g=uv^*$, then $|Dg|=|Df|$ is equivalent to 
$$
\Big(\dfrac{Dv}{v}-\dfrac{Dv^*}{v^*}\Big)\Big(\dfrac{Du^*}{u^*}-\dfrac{Du}{u}\Big)=\dfrac{DfDf^*-DgDg^*}{ff^*}=0,
$$
which imposes strong restrictions on either $u$ or $v$. With these, McDonald was able to significantly reduce the solution set into two solutions. As a consequence of Corollary \ref{cor:uv-strip}, McDonald's result directly extends to the wide-band case. We omit the proof as it is {\it mutatis mutandis} the one provided by McDonald.
\begin{corollary}
Let $f,g\in H^2_{\tau}(\mathcal{S})$, $\dfrac{\d}{\d x}$ be the operator where $\dfrac{\d}{\d x}f(x)=f'(x)$ 
for all $x\in\R$, and $D$ be one of the operators $\dfrac{\d}{\d x},\delta$ or $\gamma$.
Suppose that $|g(x)|=|f(x)|$ and $|Dg(x)|=|Df(x)|$ for $x\in\R$. Then:
	\begin{enumerate}
	\renewcommand{\theenumi}{\roman{enumi}}
		\item For the cases $D=\dfrac{\d}{\d x}$ and $D=\gamma$, either $g=\beta f$ or $g=\beta f^*$ for some constant $\beta\in \R$.
		\item For the case $D=\delta$, either $g=Vf$ or $g=Vf^*$ where $V$ is a meromorphic function that has period $b$ and continuous and unimodular on $\R$.
	\end{enumerate}
\end{corollary}

\subsection{Modulus on a segment on $\mathcal{S}$}
In the spirit of what was done by Boche {\it et. al.} \cite{Bo}, we now consider that $|g(z)|=|f(z)|$ for $z$ in a curve on $\mathcal{S}$. A similar idea can also be found in \cite{Ja2}.
For this part, we add the fact that $|g(z)|=|f(z)|$ for every $z$ on a segment lying on the strip $\mathcal{S}$.  We first look at this additional constraint on the phase retrieval problem on the disc.

\begin{lemma}
	\label{lem:angledisc}
	Let $f,g\in H^2(\D)$ such that $|g(x)|=|f(x)|$ for $x\in(-1,1)$ and 
	\begin{equation}
		\label{eq:angle}
		|g(z)|=|f(z)|,\qquad z\in e^{i\theta}(-1,1)
	\end{equation}
	where $\theta\notin\pi\mathbb Q$. Then $g=cf$ for some $c\in\T$.
\end{lemma}

\begin{proof}
	Let $f,g\in H^2(\D)$ and $\mathscr Z=Z(f)\triangle Z(g)$ be the symetric difference of the zero sets of $f$ and $g$ (that is the non common zeros). Since $|g(x)|=|f(x)|$ 
	for all $x\in (-1,1)$, we have $\mathscr Z=\overline{\mathscr Z}$. It clearly follows that $\mathscr Z\cap\R=\emptyset$. 
	\begin{center}
		\begin{tikzpicture}[line cap=round,line join=round,>=triangle 45,x=1.7714371433205747cm,y=1.7712698886159624cm]
		\draw[color=black] (-1.1903296619010413,0.) -- (1.207723250118793,0.);
		\foreach \x in {-1.,-0.5,0.5,1.}
		\draw[shift={(\x,0)},color=black] (0pt,-2pt);
		\draw[color=black] (0.,-1.09092864812463) -- (0.,1.2904137024168039);
		\foreach \y in {-1.,-0.5,0.5,1.}
		\draw[shift={(0,\y)},color=black] (-2pt,0pt);
		\clip(-1.1903296619010413,-1.09092864812463) rectangle (1.207723250118793,1.2904137024168039);
		\draw [line width=1.pt,dash pattern=on 2pt off 2pt,color=ffqqqq,fill=ffqqqq,fill opacity=0.10000000149011612] (0.,0.) ellipse (1.7714371433205747cm and 1.7712698886159624cm);
		\draw [color=ffqqqq](-0.4801049318603239,0.5718332036569326) node[anchor=north west] {$\Large\mathbb D$};
		\draw (0.021230171697829633,0.07049797196399908) node[anchor=north west] {$0$};
		\draw (0.021230171697829633,1.3572583999758616) node[anchor=north west] {$i$};
		\draw[line width=1.pt, smooth,samples=100,domain=0.0:1.0] plot[parametric] function{0.3*cos((t)),0.3*sin((t))};
		\draw (0.29696447865481407,0.37129911097975915) node[anchor=north west] {$\theta$};
		\draw [line width=1.pt,color=qqqqff] (0.5403023058681398,0.8414709848078965)-- (-0.5403023058681398,-0.8414709848078965);
		\draw (1.0239003788141365,0.0788535591588813) node[anchor=north west] {$1$};
		\draw [line width=1.pt] (1.,0.)-- (-1.,0.);
		\begin{scriptsize}
		\draw [fill=qqqqff] (0.,0.) circle (2.0pt);
		\draw [color=qqqqff] (0.5403023058681398,0.8414709848078965) circle (2.0pt);
		\draw [color=qqqqff] (-0.5403023058681398,-0.8414709848078965) circle (2.0pt);
		\draw [color=black] (1.,0.) circle (2.0pt);
		\draw [color=uuuuuu] (-1.,0.) circle (2.0pt);
		\end{scriptsize}
		\end{tikzpicture}

	\smallskip
	
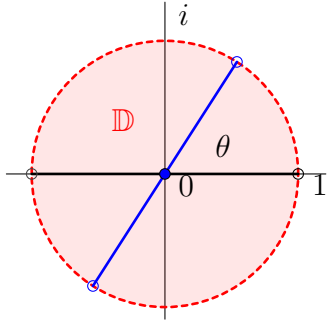
	The disc $\D$ and the segment $e^{i\theta}(-1,1)$.
	\end{center}
	
	\noindent Since $|g(x)|=|f(x)|$ 
	for all $x\in e^{i\theta}(-1,1)$, we have $\mathscr Z=\text{Ref}_\theta\mathscr Z$ where $\text{Ref}_\theta$ refers to a reflection with respect to the segment $e^{i\theta}(-1,1)$.
	Hence, by composing $\mathscr Z=\overline{\mathscr Z}$ and $\mathscr Z=\text{Ref}_\theta\mathscr Z$, we get that $\mathscr Z=\text{Rot}_{2\theta}\mathscr Z$, where $\text{Rot}_{2\theta}$ refers to a counterclockwise $2\theta$-rotation with respect to 0. Now, since $\theta\notin \pi\mathbb Q$, either $\mathscr Z=\emptyset$ or $\mathscr Z$ is uncountable. Since the zero set is discrete, $\mathscr Z $ cannot be uncountable, and so $\mathscr Z=\emptyset$. Hence, $Z(f)=Z(g)$, which implies that the Blaschke products formed by the zeros of $f$ and $g$ given by $B_f$ and $B_g$ respectively, are equal.

	Now, observe that since $|g(x)|=|f(x)|$ 
	for all $x\in (-1,1)$, Lemma \ref{lem:disc} implies that for $e^{i\zeta}\in\T$,
	$$
	\nu_f(e^{i\zeta})+\nu_f(e^{-i\zeta})=	\nu_g(e^{i\zeta})+\nu_g(e^{-i\zeta}).
	$$
	Using this equation, the Fourier coefficients of $\nu_f$ and $\nu_g$ satisfy
	\begin{equation}
	\label{eq:mfourier}
	\widehat{\nu_f} (n)+\widehat{\nu_f}(-n)=	\widehat{\nu_g}(n)+\widehat{\nu_g}(-n), \qquad n\in\mathbb N.
	\end{equation}
	On the other hand, $|g(x)|=|f(x)|$ 
	for all $x\in e^{i\theta}(-1,1)$ implies that $|f(e^{i\theta}x)|=|g(e^{i\theta}x)|$ for all $x\in (-1,1)$. For $z\in\D$, we now write $F(z)=f(e^{i\theta}z)$ and $G(z)=g(e^{i\theta}z)$ so that $F,G\in H^2(\D)$ and $|F(w)|=|G(w)|$ for all $w\in(-1,1)$. Note that for $w,z\in \D$ such that $w=ze^{i\theta}$, we have
	$$
	S_F(w)=\exp\Big(\int_{\mathbb T}\dfrac{w+e^{i\zeta}}{w-e^{i\zeta}}~\d\nu_F(e^{i\zeta})\Big)
	=\exp\Big(\int_{\mathbb T}\dfrac{ze^{i\theta}+e^{i\zeta}}{ze^{i\theta}-e^{i\zeta}}~\d\nu_F(e^{i\zeta})\Big),
	$$
	and so by letting $u=\zeta-\theta$, we get
	$$S_F(w)=\exp\Big(\int_{\mathbb T}\dfrac{z+e^{iu}}{z-e^{iu}}~\d\nu_f(e^{i(u+\theta)})\Big).
	$$
	Thus by Lemma \ref{lem:disc}, we have for $e^{i\zeta}\in\T$,
	\begin{equation}
	\label{eq:measf-t}
	\nu_f(e^{i(\theta+\zeta)})+\nu_f(e^{i(\theta-\zeta)})=	\nu_g(e^{i(\theta+\zeta)})+\nu_g(e^{i(\theta-\zeta)}).
	\end{equation}
	Next, define the measure $\mu$ on $\T$ by $\mu(e^{i\zeta})=\nu_f(e^{i(\zeta+\theta)})$ for $e^{i\zeta}\in\T$, with Fourier coefficients given by
	$$
	\widehat{\mu}(n)=\int_{\mathbb T}e^{-in\zeta}~\d\nu_f(e^{i(\zeta+\theta)})=e^{in\theta}\widehat{\nu_f}(n)
	$$
	for $n\in\mathbb N$. Hence, the previous equation and \eqref{eq:measf-t} imply that for $n\in\mathbb{N}$,
	$$
	e^{in\theta}\widehat{\nu_f} (n)+e^{-in\theta}\widehat{\nu_f}(-n)=	e^{in\theta}\widehat{\nu_g}(n)+e^{-in\theta}\widehat{\nu_g}(-n).
	$$
	Now this equation together with \eqref{eq:mfourier} imply that
	$$
	\widehat{\nu_g}(n)=\dfrac{e^{-in\theta}\widehat{\nu_f}(n)-e^{in\theta}
		\widehat{\nu_f}(n)}{e^{-in\theta}-e^{in\theta}}=\widehat{\nu_f}(n)
	$$
	and $\widehat{\nu_g}(-n)=\widehat{\nu_f}(-n)$, for all $n\in\mathbb N$. It follows that $\nu_f=\nu_g$ and so $S_f=S_g$.
	
	We now prove the same for the outer part. Since $|g(x)|=|f(x)|$ 
	for all $x\in (-1,1)$, Lemma \ref{lem:disc} again implies that for a.e. $e^{i\zeta}\in \T$,
	$$
	\log|f(e^{i\zeta})|+\log|f(e^{-i\zeta})|=\log|g(e^{i\zeta})|+\log|g(e^{-i\zeta})|.
	$$ 
	For $e^{i\zeta}\in\T$, letting $h_f(e^{i\zeta})=\log|f(e^{i\zeta})|$ implies that the Fourier coefficients of $h_f$ and $h_g$ satisfy
	\begin{equation}
	\label{eq:hf}
	\widehat{h_f} (n)+\widehat{h_f}(-n)=	\widehat{h_g}(n)+\widehat{h_g}(-n), \qquad n\in\mathbb N.
	\end{equation}
	On the other hand, by definition of $F$ and $G$, we have for a.e. $e^{i\zeta}\in\T$,
	$$
	\log|f(e^{i(\theta+\zeta)})|+\log|f(e^{i(\theta-\zeta)})|=\log|g(e^{i(\theta+\zeta)})|+\log|g(e^{i(\theta-\zeta)})|.
	$$ 
	Using this equation and a similar argument to the one for the Fourier coefficients of the singular measures, we get that for $n\in\mathbb{N}$,
	$$
	e^{in\theta}\widehat{h_f} (n)+e^{-in\theta}\widehat{h_f}(-n)=	e^{in\theta}\widehat{h_g}(n)+e^{-in\theta}\widehat{h_g}(-n).
	$$
	Hence, by this equation and \eqref{eq:hf} we get that $\widehat{h_g}(n)=\widehat{h_f}(n)$ for all $n\in\mathbb Z$. Therefore $h_f=h_g$, and so $O_f=O_g$.
	
	Finally, since $B_f=B_g$, $S_f=S_g$ and $O_f=O_g$, we have $g=cf$ for some $c\in\T$.
\end{proof}
	We now consider the coupled phase retrieval problem on the strip that includes a more general form of the constraint given in \eqref{eq:angle}. Using the previous lemma, we establish the uniqueness of the solution of the following problem.

\begin{theorem}
	\label{th:anglestrip}
	Let $f,g\in H^2_\tau(\mathcal{S})$ such that $|g(x)|=|f(x)|$ for $x\in\R$ and
	$$
	|g(z)|=|f(z)|, \quad z\in (-e^{i\theta}+a, e^{i\theta}+a)
	$$
	where $a\in\R$ and $\theta\notin\pi\mathbb Q$. Then $g=cf$ for some $c\in\T$.
\end{theorem}

\begin{proof}
	Without loss of generality, we let $a=0$ so that the segment intersects the real line at the origin. Consider $f_{1/2}(z)=f(\frac{1}{2} z),\, g_{1/2}(z)=g(\frac{1}{2} z)$ for all $z\in\D$. Observe that $f_{1/2},g_{1/2}\in H^2(\D)$, and $|g_{1/2}|=|f_{1/2}|$ on $(-1,1)$ and on $e^{i\theta}(-1,1)$. Hence, $g_{1/2}=cf_{1/2}$ on $\D$ for some $c\in \T$ by the Lemma \ref{lem:angledisc}, and so $g=cf$ on $\frac{1}{2}\D$. Therefore, since $f,g\in \mathrm{Hol}(\mathcal{S})$ and $g=cf$ on $\frac{1}{2}\D$ so we have $g=cf$ on $\mathcal{S}$.
\end{proof}
%

\subsection*{Acknowledgements}
This study has been carried out with financial support from the French State, managed
by the French National Research Agency (ANR) in the frame of the ``Investments for
the Future” Programme IdEx Bordeaux -CPU (ANR-10-IDEX-03-02).

This paper was completed during the first author's visit at the Schr\"odinger Institute, Vienna, during the workshop ``Operator Related Function Theory''. We kindly acknowledge ESI's hospitality.

The research of the second author is partially supported by the project ANR-18-CE40-0035 and the Joint French-Russian Research Project PRC-CNRS/RFBR 2017-2019.

The third author is supported by the CHED-PhilFrance scholarship from Campus France and the Commission of Higher Education (CHED), Philippines.

We would like to thank the referees for their helpful comments and suggestions. We truly appreciate the time they spent to check for corrections. Some results in this paper have been announced in \cite{JKP}. We also thank the referees of that announcement for their helpful comments that also led to improvements here.

\IfFileExists{\jobname.bbl}{}
 {\typeout{}
  \typeout{******************************************}
  \typeout{** Please run "bibtex \jobname" to optain}
  \typeout{** the bibliography and then re-run LaTeX}
  \typeout{** twice to fix the references!}
  \typeout{******************************************}
  \typeout{}
 }

\end{document}